\documentclass[11pt]{article}

\usepackage[margin=1in]{geometry}
\usepackage{amsmath,amsfonts,amsthm,amssymb}
\usepackage{subcaption}
\usepackage{graphicx}
\graphicspath{{figures/}}
\usepackage{color}
\usepackage{url}
\usepackage{soul}

\usepackage{verbatim}
\usepackage{hyperref}

\newcommand{\Rn}{\mathbb{R}^n}

\newcommand{\bx}{x}
\newcommand{\by}{y}
\newcommand{\bz}{z}

\newcommand{\bu}{{\bf u}}

\newcommand{\bb}{{\bf b}}

\newcommand{\Div}{\nabla\cdot}

\newcommand{\LLb}{\mathcal{L}_b}
\newcommand{\LLs}{\mathcal{L}_s}
\newcommand{\LLdel}{\mathcal{L}^{\delta,\beta}} 
\newcommand{\Ldel}{L^{\delta,\beta}}

\newcommand{\cdel}{c^{\delta,\beta}}
\newcommand{\cdelb}{c^{\delta,\beta+2}}

\newcommand{\mdel}{m^{\delta,\beta}}
\newcommand{\mdelb}{m^{\delta,\beta+2}}

\newcommand{\Mdel}{M^{\delta,\beta}}

\DeclareMathOperator{\trace}{trace}

\newcommand{\Ldelinf}{L^{\delta,-\infty}}

\newcommand{\Nav}{{\mathcal{N}}} 
\newcommand{\T}{{\bf T}}

\newcommand{\limd}{\lim_{\delta\rightarrow 0^+}}

\newcommand{\Bd}{\partial_\delta}

\newcommand{\lam}{\lambda^{*}}

\newtheorem{theorem}{Theorem}
\newtheorem{lemma}{Lemma\textbf{}}
\newtheorem{cor}{Corollary}
\newtheorem{prop}{Proposition}

\theoremstyle{remark}
\newtheorem{remark}{Remark}

\title{Linear peridynamics Fourier multipliers and eigenvalues}
\author{Bacim Alali and Nathan Albin\\
\\
{\footnotesize Department of Mathematics, Kansas State University, Manhattan, KS}}

\begin{document}

\maketitle

\begin{abstract}
A characterization for the Fourier multipliers and eigenvalues of linear peridynamic operators is provided. The analysis is presented for  state-based peridynamic operators for isotropic homogeneous media in any spatial dimension. We provide explicit formulas for the eigenvalues in terms of the space dimension, the nonlocal parameters, and the material properties.

The approach we follow is based on the Fourier multiplier analysis developed in \cite{alali2019fourier}. The Fourier multipliers of linear peridynamic operators are second-order tensor fields, which are given through integral representations.
It is shown that the eigenvalues of the peridynamic operators can be derived directly from the eigenvalues of the Fourier multiplier tensors.
We reveal a simple structure for the Fourier multipliers in terms of hypergeometric functions, which allows for providing integral representations as well as hypergeometric representations of the eigenvalues. These representations are utilized to show  the convergence of the eigenvalues of linear peridynamics to the eigenvalues of the Navier operator of linear elasticity in the limit of vanishing nonlocality. Moreover, the hypergeometric representation of the eigenvalues is utilized to compute the spectrum of linear peridynamic operators.
 
\end{abstract}


\section{Introduction}
In this work, we study the Fourier multipliers of  linear state-based peridynamic operators. The main goals are to find explicit representations for the multipliers, when the operator is defined on $\Rn$,  and to find explicit representations for the eigenvalues of the peridynamic operator, when it is defined on periodic domains. The  formulas that we derive for the Fourier multipliers and the eigenvalues
are of two types: nonlocal (integral) representations and representations in terms of hypergeometric functions. As have been demonstrated  in \cite{alali2019fourier} and \cite{alali2020fourier}, such explicit representations can be exploited to rigorously characterize the behavior of nonlocal operators and develop regularity theory for nonlocal equations, as well as to devise efficient and accurate spectral methods for the numerical solutions of nonlocal equations. The current work  focuses on 
the derivations of these representations, while the regularity of peridynamic equations and spectral methods for peridynamics based on the approach presented here will be pursued in forthcoming  works.

There has been a recent increased interest in spectral methods for peridynamics and nonlocal equations as these methods provide efficient and accurate solvers.
One of the features of these spectral solvers  is that the nonlocality parameters 
do not scale with the grid size, thus providing computational accuracy and efficiency \cite{alali2020fourier}. Spectral methods have been developed for nonlocal and peridynamic equations in periodic  domain,   bounded domains, and for problems on surfaces, as well as problems involving fracture \cite{du2017fast}, \cite{slevinsky2018spectral}, \cite{jafarzadeh2021fast}, and \cite{jafarzadeh2022general}.
Spectral and Fourier multipliers approaches  provide analysis techniques for studying the regularity of solutions of nonlocal equations, see for example \cite{du2016asymptotically} and \cite{alali2019fourier}. The work in \cite{scott2022fractional} follows a Fourier multipliers approach to study a fractional  Lam\'e-Navier operator, its connection to state-based peridynamics, and to establish analysis results for this operator and certain associated fractional equations, see also \cite{scott2020mathematical}.    

The approach presented in this work to uncover explicit formulas for the multipliers and the eigenvalues is based on two indirect connections; the first is a connection between the multipliers of the peridynamic operator, which are second-order tensor fields, and the scalar multipliers of the nonlocal Laplace operator. The second connection is between the multipliers of the peridynamic operator, defined on $\Rn$, and the eigenvalues of the peridynamic operator, when it is defined on periodic domains. 
Throughout this article we refer to the Fourier multipliers of the nonlocal Laplacian as the {\it scalar multipliers}, whereas the {\it tensor multipliers} refer to the Fourier multipliers of the peridynamic operator.  

A brief description of the main steps in our approach and the organization of the article are as follows. The definition of the linear peridynamic operator in $\Rn$ and the specific integral kernels are provided in Section~\ref{sec:overview}. In order to find explicit representations in terms of the nonlocality parameters and the space dimension, we focus on integral kernels of the form \eqref{eq:kernel}, which can be singular or integrable.
 However, we emphasize that the results in this work can be generalized to other types of integral kernels. Section~\ref{sec:scalar-multipliers} presents the nonlocal Laplacian and its  multipliers given by  the integral and hypergeometric representations \eqref{eq:multiplier-cosine} and \eqref{eq:multiplier-general}, respectively. The multipliers of the Navier operator of linear elasticity and the integral formula for the tensor multipliers of the peridynamic operator are derived in Section~\ref{sec:integral-tensor-multipliers}. 
 Each entry of the $n\times n$ tensor multiplier is written as an integral in $\Rn$. A key step in our approach is to reveal a simple structure for this tensor. This is accomplished in Section~\ref{sec:structure}, where we show in Section~\ref{sec:derivatives} that the tensor multipliers can be recovered using the derivatives of the scalar multipliers. By combining this relationship with the hypergeometric formula of the scalar multipliers together with the aid of some  facts about hypergeometric functions as presented in Section~\ref{sec:hypergeometric-formulas}, we arrive at a simple structure for the tensor multipliers in terms of hypergeometric functions as demonstrated in Section~\ref{sec:tensors-hypergeometric}.    An immediate consequence of this result is the convergence of the tensor multipliers of the peridynamic operator to the tensor multipliers of the Navier operator for two kinds of local limits. In Section~\ref{sec:eigenvectors}, the tensor multiplier at any vector in $\Rn$ is shown to be a real symmetric matrix  with $n$ orthonormal eigenvectors and two distinct associated eigenvalues. Using the hypergeometric representation for the tensor multipliers, we derive explicit formulas for these eigenvalues in terms of hypergeometric functions. Using these eigenvalue formulas, we derive  integral representations for the eigenvalues in Section~\ref{sec:eigenvalues-integral}.
In Section~\ref{sec:periodic}, we consider the peridynamic operator defined for periodic vector-fields. We show how  the eigenvector fields and the eigenvalues for the peridynamic operator on periodic domains can  be derived from the tensor multipliers' eigenvectors and  eigenvalues.

\section{Overview}
\label{sec:overview}
Linear peridynamic operators defined in a domain $\Omega\subseteq\Rn$ have the form  \cite{silling2010linearized}
\begin{equation*}
    \mathcal{L}\bu(x)=\int_{\Omega}C(x,y) (\bu(y)-\bu(x))\,dy,
\end{equation*}
where $C(x,y)$ is a second-order tensor and $\bu:\Rn\rightarrow\Rn$ is a vector field. For a homogeneous isotropic solid, the linear operator takes the form
\begin{eqnarray*}
    \mathcal{L} \bu(x)&=& \rho\int_{\Omega}
\frac{\gamma(\|y-x\|)}{\|y-x\|^2}(y-x)\otimes(y-x)\big( \bu(y)-\bu(x)\big)\,dy\\
  \nonumber\\
  && + \rho'\int_{\Omega}\int_{\Omega}
 \gamma(\|y-x\|)\gamma(\|z-x\|)(y-x)\otimes(z-x)\big( \bu(z)-\bu(x)\big)\,dz dy\\
  && + \rho'\int_{\Omega}\int_{\Omega}
 \gamma(\|y-x\|)\gamma(\|z-y\|)(y-x)\otimes(z-y)\big( \bu(z)-\bu(y)\big)\,dz dy, 
\end{eqnarray*}
where $\gamma$ is a scalar field, and $\rho$ and $\rho'$ are scaling constants that include the material properties. Taking $\Omega=\Rn$, and due to symmetry, the operator reduces to
\begin{eqnarray*}
    \mathcal{L} \bu(x)&=& \rho\int_{\Rn}
\frac{\gamma(\|y-x\|)}{\|y-x\|^2}(y-x)\otimes(y-x)\big( \bu(y)-\bu(x)\big)\,dy\\
  \nonumber\\
   && + \rho'\int_{\Rn}\int_{\Rn}
 \gamma(\|y-x\|)\gamma(\|z-y\|)(y-x)\otimes(z-y) \bu(z)\,dz dy. 
\end{eqnarray*}
 In this work, we focus on radially symmetric kernels with compact support of the form
 \begin{equation}\label{eq:kernel}
 \gamma(\|y-x\|)=\cdel  \frac{1}{\|y-x\|^\beta}\; \chi_{B_\delta(x)}(y),
 \end{equation}
where $\cdel$ is given by \eqref{eq:cdel-explicit}, $\chi_{B_\delta(x)}$ is the indicator function of the ball of radius $\delta>0$ centered at $x$, and  the exponent satisfies $\beta<n+2$.
In this case, the linear peridynamic operator, parametrized by the horizon (nonlocality parameter) $\delta$ and the integral kernel exponent $\beta$, can be written as
\begin{eqnarray}
 \nonumber
    \LLdel \bu(\bx)&=& (n+2)\,\mu\,\cdel\int_{B_\delta(\bx)}
   \frac{(\by-\bx)
\otimes(\by-\bx)}{\|\by-\bx\|^{\beta+2}}\big( \bu(\by)-\bu(\bx)\big)
  \,d\by\\
  \nonumber\\
  && + (\lam-\mu)\,\frac{(\cdel)^2}{4}\int_{B_\delta(\bx)}\int_{B_\delta(\by)}
 \frac{\by-\bx}{\|\by-\bx\|^\beta}
\otimes\frac{\bz-\by}{\|\bz-\by\|^\beta}\,\bu(\bz)
 \,d\bz d\by,
  \label{eq:LLdel_homog}
\end{eqnarray}
where  $\mu$ and $\lam$ are Lam\'{e} parameters, and
the scaling constant $\cdel$ is 
defined by
\begin{eqnarray}
\label{eq:cdel-integral}
\nonumber
\cdel &:=& \left(\frac{1}{2 n}\int_{B_\delta(0)}\frac{\|w\|^2}{\|w\|^\beta}\;dw\right)^{-1},\\
\label{eq:cdel-explicit}
&=& \frac{2(n+2-\beta)\Gamma\left(\frac{n}{2}+1\right)}
    {\pi^{n/2}\delta^{n+2-\beta}}.
\end{eqnarray}

\begin{remark}
The second Lam\'{e} parameter is usually denoted by $\lambda$, but we choose to use $\lam$ instead in order to keep $\lambda$ to denote an eigenvalue.
\end{remark}

It is convenient to use the  following decomposition of  $\LLdel$
\[
\LLdel=\LLb+\LLs,
\]
where, after changing variables, 
\begin{equation}
 \LLb \bu(\bx)= (n+2)\,\mu\,\cdel\int_{B_\delta(0)}
   \frac{w\otimes w}{\|w\|^{\beta+2}}\big( \bu(x+w)-\bu(\bx)\big)\,dw,
  \label{eq:Lb}
\end{equation}
and
\begin{equation}
 \LLs \bu(\bx)= (\lam-\mu)\,\frac{(\cdel)^2}{4}\int_{B_\delta(0)}\int_{B_\delta(0)}
 \frac{w}{\|w\|^\beta}
\otimes\frac{q}{\|q\|^\beta}\,\bu(x+q+w)\,dq dw.
 \label{eq:Ls}
\end{equation}
We note that $\LLb$ is the linear operator for  bond-based peridynamics.

We denote by $\Nav$ the Navier operator  of linear elasticity. For a homogeneous isotropic medium, it is given by
 \begin{eqnarray}
\label{Nhomo}
\Nav\bu&=& (\lam+\mu)\nabla( \Div\bu) + \mu\Delta \bu.
\end{eqnarray}

\section{Fourier multipliers}\label{sec:multipliers}
\subsection{Multipliers for the nonlocal Laplacian}
\label{sec:scalar-multipliers}
For scalar fields $u:\Rn\rightarrow \mathbb{R}$, the analogue to the peridynamic operator $\LLdel$, is the nonlocal Laplacian,  which in this case is given by
\begin{equation}\label{eq:nonlocal_laplacian}
  \Ldel u(x) = \cdel\int_{B_\delta(x)}\frac{u(y)-u(x)}{\|y-x\|^\beta}\;dy,
\end{equation}
with $\cdel$ given by \eqref{eq:cdel-explicit}.

The Fourier multipliers for the nonlocal Laplacian in \eqref{eq:nonlocal_laplacian} have been studied in \cite{alali2019fourier}, in which the multiplier $\mdel$ is defined through the Fourier transform by
\begin{equation}\label{eq:L-multiplier}
	\Ldel u(x) = \frac{1}{(2\pi)^n}\int_{\mathbb{R}^n}\mdel(\nu)\widehat{u}(\nu)e^{i\nu\cdot x}\;d\nu,
\end{equation}
where   $\mdel$ has the integral representation
\begin{equation}\label{eq:multiplier-cosine}
	\mdel(\nu) = \cdel\int_{B_\delta(0)}\frac{\cos(\nu\cdot w)-1}{\|w\|^\beta}\;dw.
\end{equation}
The  hypergeometric  representation  of  the multipliers is provided by 
\begin{equation}\label{eq:multiplier-general}
	\mdel(\nu) = -\|\nu\|^2\,_2F_3\left(1,\frac{n+2-\beta}{2};2,\frac{n+2}{2},\frac{n+4-\beta}{2};-\frac{1}{4}\|\nu\|^2\delta^2\right).
\end{equation}

\subsection{Integral representations for the peridynamic multipliers}
\label{sec:integral-tensor-multipliers}
In this section, we extend the approach developed in \cite{alali2019fourier} for the nonlocal Laplacian $\Ldel$ in \eqref{eq:nonlocal_laplacian} to the peridynamic operator $\LLdel$ in \eqref{eq:LLdel_homog}. We begin by deriving integral formulas for the Fourier multipliers of $\LLdel$. 
Express $\bu$ through its Fourier transform as
\begin{equation*}
	\bu(x) = \frac{1}{(2\pi)^n}\int_{\mathbb{R}^n}\widehat{\bu}(\nu)e^{i\nu\cdot x}\;d\nu.
\end{equation*}
Since the definition of $\LLdel$ can be extended to the space of tempered distributions through the multipliers derived below, it is sufficient  to assume that $\bu$ is a Schwartz vector field.  We compute the multipliers for $\LLb$ and $\LLs$ separately. Applying $\LLb$  shows that
\begin{equation*}
\begin{split}
	\LLb \bu(x) &=  (n+2)\,\mu\,\cdel\int_{B_\delta(0)}
   \frac{w\otimes w}{\|w\|^{\beta+2}}\big( \bu(x+w)-\bu(\bx)\big)\,dw,\\
    &= \frac{1}{(2\pi)^n}\int_{\mathbb{R}^n}\left[
    (n+2)\,\mu\,\cdel \int_{B_\delta(0)}\frac{w\otimes w}{\|w\|^{\beta+2}}\left(e^{i\nu\cdot w}-1\right)\;dw
    \right]\widehat{\bu}(\nu)e^{i\nu\cdot x}\;d\nu,
\end{split}
\end{equation*}
providing the representation
\begin{equation}
	\LLb \bu(x) = \frac{1}{(2\pi)^n}\int_{\mathbb{R}^n}M_b(\nu)\widehat{\bu}(\nu)e^{i\nu\cdot x}\;d\nu,
\end{equation}
where
\begin{eqnarray}\label{eq:Mb}
\nonumber
	M_b(\nu) &=& (n+2)\,\mu\,\cdel \int_{B_\delta(0)}\frac{w\otimes w}{\|w\|^{\beta+2}}\left(e^{i\nu\cdot w}-1\right)\;dw\\
	&=&(n+2)\,\mu\,\cdel \int_{B_\delta(0)}\frac{w\otimes w}{\|w\|^{\beta+2}}\left(\cos(\nu\cdot w)-1\right)\;dw.
\end{eqnarray}
Similarly, we compute the multipliers of $\LLs$, 
\begin{equation*}
\begin{split}
	\LLs \bu(x) &=  (\lam-\mu)\,\frac{(\cdel)^2}{4}\int_{B_\delta(0)}\int_{B_\delta(0)}
 \frac{w}{\|w\|^\beta}
\otimes\frac{q}{\|q\|^\beta}\,\bu(x+q+w)\,dq dw,\\
    &= \frac{1}{(2\pi)^n}\int_{\mathbb{R}^n}\left[
   (\lam-\mu)\,\frac{(\cdel)^2}{4}\int_{B_\delta(0)}
 \frac{w}{\|w\|^\beta} e^{i \nu\cdot w}\,dw
\otimes\int_{B_\delta(0)}\frac{q}{\|q\|^\beta} e^{i \nu\cdot q}\,dq 
    \right]\widehat{\bu}(\nu)e^{i\nu\cdot x}\;d\nu,
\end{split}
\end{equation*}
providing the representation
\begin{equation}
	\LLs \bu(x) = \frac{1}{(2\pi)^n}\int_{\mathbb{R}^n}M_s(\nu)\widehat{\bu}(\nu)e^{i\nu\cdot x}\;d\nu,
\end{equation}
where
\begin{eqnarray}\label{eq:Ms}
\nonumber
	M_s(\nu) &=& (\lam-\mu)\,\frac{(\cdel)^2}{4}\left(\int_{B_\delta(0)}
 \frac{w}{\|w\|^\beta} e^{i \nu\cdot w}\,dw\right)
\otimes\left(\int_{B_\delta(0)}\frac{w}{\|w\|^\beta} e^{i \nu\cdot w}\,dw\right)\\
	&=&-(\lam-\mu)\,\frac{(\cdel)^2}{4}\left(\int_{B_\delta(0)}
 \frac{w}{\|w\|^\beta} \sin(\nu\cdot w)\,dw \right)
\otimes\left(\int_{B_\delta(0)}\frac{w}{\|w\|^\beta} \sin(\nu\cdot w)\,dw\right).
\end{eqnarray}
Combining \eqref{eq:Mb} and \eqref{eq:Ms}, we obtain the multipliers for $\LLdel$,
\begin{eqnarray}\label{eq:M}
\Mdel= M_b+M_s, 
\end{eqnarray}
which satisfy
\begin{eqnarray}
\nonumber
\widehat{\LLdel \bu}= \Mdel \,\widehat{\bu}. 
\end{eqnarray}
The  following  summarizes the results of this subsection.
\begin{prop}
The Fourier multipliers $\Mdel$ of the linear peridynamic operator $\LLdel$ in \eqref{eq:LLdel_homog} are characterized through  integral representations as given by  \eqref{eq:M}, \eqref{eq:Mb} and \eqref{eq:Ms}.
\end{prop}

We note that the Fourier multipliers of $\Nav$, the Navier operator given in \eqref{Nhomo}, are similarly defined by
\[
\widehat{\Nav \bu}= M^{\Nav} \,\widehat{\bu}, 
\]
and can be shown to be given explicitly by 
\begin{equation}
\label{eq:N-mutipliers}    
M^{\Nav}(\nu)=-(\lam+\mu) \nu\otimes\nu -\mu \|\nu\|^2 \; I,
\end{equation}
where $I$ is the identity matrix.
\subsection{Peridynamic multipliers: Structure and hypergeometric representations}
\label{sec:structure}
We  emphasize that the multipliers of linear peridynamics, given by \eqref{eq:Mb},\eqref{eq:Ms}, and \eqref{eq:M}, are second-order tensor fields. In this section, we reveal a simple and explicit structure for the matrix $\Mdel(\nu)$ in terms of $\nu$ and the derivatives of the scalar multipliers (multipliers of the nonlocal Laplacian) $\mdel(\nu)$ given by \eqref{eq:multiplier-cosine} or, equivalently, by \eqref{eq:multiplier-general}.
\subsubsection{Hypergeometric formulas}\label{sec:hypergeometric-formulas}
In this section, we derive and present hypergeometric formulas that will be useful in the subsequent sections.  Let $\mathbf{a}=(a_1,a_2,\ldots,a_p)$ and $\mathbf{b}=(b_1,b_2,\ldots,b_q)$ be two vectors of coefficients.  The generalized hypergeometric function ${}_pF_q$ with parameters $\mathbf{a}$ and $\mathbf{b}$ is defined as
\begin{equation*}
{}_pF_q(\mathbf{a};\mathbf{b};z) := \sum_{k=0}^\infty\frac{(\mathbf{a})_k}{(\mathbf{b})_k}\frac{z^k}{k!}.
\end{equation*}
Here, the notation $(\mathbf{a})_k$ represents the product
\begin{equation*}
(\mathbf{a})_k = (a_1)_k(a_2)_k\cdots(a_p)_k,
\end{equation*}
where $(a)_k$ is the Pochhammer symbol
\begin{equation*}
(a)_k = \frac{\Gamma(a+k)}{\Gamma(a)} = a(a+1)(a+2)\cdots(a+k-1).
\end{equation*}
We also define the notation
\begin{equation*}
\prod\mathbf{a} = a_1a_2\cdots a_p\qquad\text{and}\qquad\mathbf{a}+c = (a_1+c,a_2+c,\ldots,a_p+c),
\end{equation*}
and recall the following useful facts about the Pochhammer symbol.
\begin{equation}\label{eq:poch-relation}
(a)_{k+1} = a(a+1)\cdots(a+k-1)(a+k) = a(a+1)_k,
\end{equation}
and
\begin{equation}\label{eq:poch-ratio}
\frac{(a+1)_k}{(a)_k} = \frac{a+k}{a}.
\end{equation}

In light of~\eqref{eq:multiplier-general}, we consider the derivatives of a function of the form
\begin{equation}\label{eq:hypergeometric-form}
f(z) = z\cdot{}_pF_q(\mathbf{a};\mathbf{b};z)
= z\sum_{k=0}^\infty\frac{(\mathbf{a})_k}{(\mathbf{b})_k}\frac{z^k}{k!}
= \sum_{k=0}^\infty\frac{(\mathbf{a})_k}{(\mathbf{b})_k}\frac{z^{k+1}}{k!}.
\end{equation}

\begin{lemma}\label{lem:hypergeometric-derivatives}
Let $f(x)$ have the form~\eqref{eq:hypergeometric-form}.  Then
\begin{equation}\label{eq:f-prime}
f'(z) = {}_{p+1}F_{q+1}(\mathbf{a}';\mathbf{b}';z)
\end{equation}
and
\begin{equation}\label{eq:f-prime2}
f''(z) = \frac{\prod\mathbf{a'}}{\prod\mathbf{b'}}\;{}_{p+1}F_{q+1}(\mathbf{a}'+1;\mathbf{b}'+1;z),
\end{equation}
where
\begin{equation*}
\mathbf{a'}= (2,a_1,\ldots,a_p)\qquad\text{and}\qquad
\mathbf{b'}= (1,b_1,\ldots,b_q).
\end{equation*}
\end{lemma}

\begin{proof}
Taking the term-wise first derivative and applying~\eqref{eq:poch-ratio} shows that
\begin{equation*}
f'(z) = \sum_{k=0}^\infty\frac{(\mathbf{a})_k(k+1)}{(\mathbf{b})_k}\frac{z^{k}}{k!}
= \sum_{k=0}^\infty\frac{(\mathbf{a})_k(2)_k}{(\mathbf{b})_k(1)_k}\frac{z^{k}}{k!}
= {}_{p+1}F_{q+1}(\mathbf{a}';\mathbf{b}';z),
\end{equation*}

Taking a term-wise derivative once again, then reindexing and using~\eqref{eq:poch-relation} yields
\begin{equation*}
\begin{split}
f''(z) &= \sum_{k=1}^\infty\frac{(\mathbf{a'})_kk}{(\mathbf{b'})_k}\frac{z^{k-1}}{k!}
= \sum_{k=0}^\infty\frac{(\mathbf{a'})_{k+1}(k+1)}{(\mathbf{b'})_{k+1}}\frac{z^{k}}{(k+1)!}\\
&= \frac{\prod\mathbf{a'}}{\prod\mathbf{b'}}\sum_{k=0}^\infty\frac{(\mathbf{a'}+1)_{k}}{(\mathbf{b'}+1)}\frac{z^{k}}{k!}
= \frac{\prod\mathbf{a'}}{\prod\mathbf{b'}}\;{}_{p+1}F_{q+1}(\mathbf{a}'+1;\mathbf{b}'+1;z).
\end{split}
\end{equation*}
\end{proof}

Two additional formulas we shall use are found in the following lemmas.

\begin{lemma}\label{lem:pFq-minus-1}
For any choice of coefficients,
\begin{equation}\label{eq:pFq-minus-1}
{}_pF_q(\mathbf{a};\mathbf{b};z) - 1 =
\frac{\prod\mathbf{a}}{\prod\mathbf{b}}z\cdot{}_{p+1}F_{q+1}(1,\mathbf{a}+1;2,\mathbf{b}+1);z).
\end{equation}
\end{lemma}

\begin{proof}
This is again found by term-wise differentiation, reindexing and applying~\eqref{eq:poch-relation}.
\begin{equation*}
\begin{split}
{}_pF_q(\mathbf{a};\mathbf{b};z) - 1 &= \sum_{k=1}^\infty\frac{(\mathbf{a})_k}{(\mathbf{b})_k}\cdot\frac{z^k}{k!}
= \sum_{k=0}^\infty\frac{(\mathbf{a})_{k+1}}{(\mathbf{b})_{k+1}}\cdot\frac{z^{k+1}}{(k+1)!}\\
&= \frac{\prod\mathbf{a}}{\prod\mathbf{b}}z\
\sum_{k=0}^\infty\frac{(\mathbf{a}+1)_{k}(1)_k}{(\mathbf{b}+1)_{k}(2)_k}\cdot\frac{z^{k}}{k!}.
\end{split}
\end{equation*}
\end{proof}

\begin{lemma}\label{lem:pFq-linear-combination}
For any choice of hypergeometric coefficients and for any numbers $c$ and $d$,
\begin{equation}\label{eq:pFq-linear-combination}
c\;{}_{p+1}F_{q+1}(1,\mathbf{a};2,\mathbf{b};z) + d\;{}_{p}F_{q}(\mathbf{a};\mathbf{b};z) =
(c+d)\;{}_{p+2}F_{q+2}\left(1,\frac{c+2d}{d},\mathbf{a};2,\frac{c+d}{d},\mathbf{b};z\right).
\end{equation}
\end{lemma}

\begin{proof}
This can be seen by term-wise addition and using~\eqref{eq:poch-ratio}:

\begin{equation*}
\begin{split}
c\;{}_{p+1}F_{q+1}(1,\mathbf{a};2,\mathbf{b};z) + d\;{}_{p}F_{q}(\mathbf{a};\mathbf{b};z) &= 
\sum_{k=0}^\infty\left(
\frac{c(\mathbf{a})_k(1)_k}{(\mathbf{b})_k(2)_k}
+ \frac{d(\mathbf{a})_k}{(\mathbf{b})_k}
\right)\frac{z^k}{k!}\\
&=
\sum_{k=0}^\infty\left(
\frac{c+dk+d}{k+1}
\right)\frac{(\mathbf{a})_k}{(\mathbf{b})_k}\frac{z^k}{k!}\\
&=
d\sum_{k=0}^\infty
\frac{k+\frac{c+d}{d}}{k+1}
\frac{(\mathbf{a})_k}{(\mathbf{b})_k}\frac{z^k}{k!}\\
&=
(c+d)\sum_{k=0}^\infty
\frac{\left(\frac{c+2d}{d}\right)_k(1)_k}{\left(\frac{c+d}{d}\right)_k(2)_k}
\frac{(\mathbf{a})_k}{(\mathbf{b})_k}\frac{z^k}{k!}.
\end{split}
\end{equation*}
\end{proof}

\subsubsection{Derivatives of the scalar multipliers}
\label{sec:derivatives}
In this section, we show how the tensor multipliers $\Mdel$, and in particular $M_b$ and $M_s$, can be recognized in terms of the derivatives of the scalar multipliers $\mdel$.

Differentiating $\mdel$ in \eqref{eq:multiplier-cosine} with respect to $\nu_i$ shows that
\begin{equation}\label{eq:dm-dnu}
\frac{\partial}{\partial\nu_i}\mdel(\nu) = \cdel\int_{B_\delta(0)}\frac{-w_i\sin(\nu\cdot w)}{\|w\|^{\beta}}\;dw.
\end{equation}
Substituting this into~\eqref{eq:Ms} yields the formula
\begin{equation}\label{eq:M_s-diff-fmla}
\left(M_s(\nu)\right)_{ij} = -\frac{\lam-\mu}{4}\frac{\partial}{\partial\nu_i}\mdel(\nu)\frac{\partial}{\partial\nu_j}\mdel(\nu).
\end{equation}
Differentiating a second time in~\eqref{eq:dm-dnu} (and replacing $\beta$ by $\beta+2$), yields
\begin{equation}\label{eq:dm-dnu2}
\frac{\partial^2}{\partial\nu_i\partial\nu_j}\mdelb(\nu) = \cdelb\int_{B_\delta(0)}\frac{-w_i w_j\cos(\nu\cdot w)}{\|w\|^{\beta+2}}\;dw,
\end{equation}
which implies that
\begin{equation*}
\int_{B_\delta(0)}\frac{w_i w_j\cos(\nu\cdot w)}{\|w\|^{\beta+2}}\;dw = -(\cdelb)^{-1}\frac{\partial^2}{\partial\nu_i\partial\nu_j}m^{\delta,\beta+2}(\nu).
\end{equation*}
Moreover,
\begin{equation}
\int_{B_\delta(0)}\frac{w_i w_j}{\|w\|^{\beta+2}}\;dw 
=\delta_{ij}\int_{B_\delta(0)}\frac{w_i^2}{\|w\|^{\beta+2}}\;dw
= 2(\cdelb)^{-1}\delta_{ij}.\label{eq:identity}
\end{equation}
Substituting these last two formulas into~\eqref{eq:Mb} shows that
\begin{equation}\label{eq:M_b-diff-fmla}
\left(M_b(\nu)\right)_{ij} = -\frac{(n+2)\mu \cdel}{\cdelb}
\left(
\frac{\partial^2}{\partial\nu_i\partial\nu_j}m^{\delta,\beta+2}(\nu) + 2\delta_{ij}
\right).
\end{equation}

The scalar multipliers $\mdel$ can be written as
\begin{equation}\label{eq:m-hypergeom-2}
\begin{split}
\mdel(\nu) &= -\|\nu\|^2
{}_2F_3\left(
1,\frac{n+2-\beta}{2};2,\frac{n+2}{2},\frac{n+4-\beta}{2};-\frac{1}{4}\|\nu\|^2\delta^2
\right)\\
&= \frac{4}{\delta^2}\left(-\frac{1}{4}\|\nu\|^2\delta^2\right)
{}_2F_3\left(
1,\frac{n+2-\beta}{2};2,\frac{n+2}{2},\frac{n+4-\beta}{2};-\frac{1}{4}\|\nu\|^2\delta^2
\right)\\
&= \frac{4}{\delta^2}f\left(-\frac{1}{4}\|\nu\|^2\delta^2\right)
\end{split}
\end{equation}
where $f$ has the form~\eqref{eq:hypergeometric-form} with $p=2$, $q=3$ and coefficients $\mathbf{a}$ and $\mathbf{b}$ defined to match~\eqref{eq:m-hypergeom-2}. Differentiating once shows that
\begin{equation}\label{eq:pd-m-f}
\frac{\partial}{\partial\nu_i} \mdel(\nu)
= \frac{4}{\delta^2}f'\left(-\frac{1}{4}\|\nu\|^2\delta^2\right)\cdot\left(-\frac{1}{2}\delta^2\nu_i\right)
= -2f'\left(-\frac{1}{4}\|\nu\|^2\delta^2\right)\nu_i.
\end{equation}
Differentiating a second time shows that
\begin{equation}\label{eq:pd2-m-f}
\begin{split}
\frac{\partial^2}{\partial\nu_i\partial\nu_j} \mdel(\nu)
&= -2f'\left(-\frac{1}{4}\|\nu\|^2\delta^2\right)\delta_{ij}
-2f''\left(-\frac{1}{4}\|\nu\|^2\delta^2\right)\nu_i\cdot\left(-\frac{1}{2}\delta^2\nu_j\right)\\
&= -2f'\left(-\frac{1}{4}\|\nu\|^2\delta^2\right)\delta_{ij}
+ \delta^2f''\left(-\frac{1}{4}\|\nu\|^2\delta^2\right)\nu_i\nu_j.
\end{split}
\end{equation}

The results of this subsection are summarized by the following.
\begin{prop} The tensor multipliers $M_b$ and $M_s$ can be represented in terms of the gradients of the scalar multipliers $\mdel$ as 
\begin{equation*}
\begin{split}
M_b(\nu) &= -\frac{(n+2)\;\mu\; \cdel}{\cdelb}
\left(
\nabla\nabla m^{\delta,\beta+2}(\nu) + 2 I
\right),\\
&= -\frac{(n+2)\mu \cdel}{\cdelb}
\left(\delta^2f''\left(-\frac{1}{4}\|\nu\|^2\delta^2\right)\nu\otimes \nu + \left(2-2f'\left(-\frac{1}{4}\|\nu\|^2\delta^2\right)\right) I\right),
\end{split}
\end{equation*}
where  $f$ has the form~\eqref{eq:hypergeometric-form} with $p=2$, $q=3$ and coefficients $\mathbf{a}=\left(1,\frac{n-\beta}{2}\right)$ and $\mathbf{b}=\left(2,\frac{n+2}{2},\frac{n+2-\beta}{2}\right)$, 
and,
\begin{equation*}
\begin{split}
M_s(\nu) &=  -\frac{\lam-\mu}{4}\;\nabla\mdel(\nu)\otimes\nabla\mdel(\nu),\\
&= -(\lam-\mu)\left(f'\left(-\frac{1}{4}\|\nu\|^2\delta^2\right)\right)^2 \; \nu\otimes \nu.
\end{split}
\end{equation*}
where $f$ has the form~\eqref{eq:hypergeometric-form} with $p=2$, $q=3$ and coefficients $\mathbf{a}=\left(1,\frac{n+2-\beta}{2}\right)$ and $\mathbf{b}=\left(2,\frac{n+2}{2},\frac{n+4-\beta}{2}\right)$.
\end{prop}
This result together with the formulas derived in Section~\ref{sec:hypergeometric-formulas} allow us to express the tensor multipliers as hypergeometric functions.

\subsubsection{The  tensor multipliers: Hypergeometric representations}
\label{sec:tensors-hypergeometric}
In this section, we provide a simple and explicit form for the tensor multipliers $M_b(\nu)$ and $M_s(\nu)$.

Equations~\eqref{eq:M_s-diff-fmla},~\eqref{eq:pd-m-f} and Lemma~\ref{lem:hypergeometric-derivatives} show that $M_s(\nu)$ is a rank-one symmetric matrix of the form
\begin{equation}
M_s(\nu) = \alpha_s(\nu)\nu\otimes\nu,
\label{eq:M_s2}
\end{equation}
with
\begin{equation}
\begin{split}
\alpha_s(\nu) &= -\frac{\lam-\mu}{4} \;\left(-2f'\left(-\frac{1}{4}\|\nu\|^2\delta^2\right)\right)^2\\
&= -(\lam-\mu)\left(f'\left(-\frac{1}{4}\|\nu\|^2\delta^2\right)\right)^2\\
&= -(\lam-\mu)
{}_3F_4\left(
1,2,\frac{n+2-\beta}{2};1,2,\frac{n+2}{2},\frac{n+4-\beta}{2};-\frac{1}{4}\|\nu\|^2\delta^2
\right)^2\\
&= -(\lam-\mu)
{}_1F_2\left(
\frac{n+2-\beta}{2};\frac{n+2}{2},\frac{n+4-\beta}{2};-\frac{1}{4}\|\nu\|^2\delta^2
\right)^2.
\end{split}\label{eq:alpha_s}
\end{equation}

Equations~\eqref{eq:M_b-diff-fmla},~\eqref{eq:pd2-m-f} and Lemma~\ref{lem:hypergeometric-derivatives} show that $M_b(\nu)$ is a symmetric matrix of the form
\begin{equation}
M_b(\nu) = \alpha_{b1}(\nu)I + \alpha_{b2}(\nu)\nu\otimes\nu.\label{eq:M_b2}
\end{equation}
The coefficient of the identity matrix is (keeping in mind that this time we are differentiating $\mdelb$),
\begin{equation}
\begin{split}
\alpha_{b1}(\nu) &= -\frac{(n+2)\mu \cdel}{\cdelb}
\left(2-2f'\left(-\frac{1}{4}\|\nu\|^2\delta^2\right)\right)\\
&= \frac{2(n+2)\mu \cdel}{\cdelb}\left(
{}_3F_4\left(
1,2,\frac{n-\beta}{2};1,2,\frac{n+2}{2},\frac{n+2-\beta}{2};-\frac{1}{4}\|\nu\|^2\delta^2
\right)
-1\right)\\
&= \frac{2(n+2)\mu \cdel}{\cdelb}\left(
{}_1F_2\left(
\frac{n-\beta}{2};\frac{n+2}{2},\frac{n+2-\beta}{2};-\frac{1}{4}\|\nu\|^2\delta^2
\right)
-1\right).
\end{split}\label{eq:alpha_b1-1}
\end{equation}
Applying Lemma~\ref{lem:pFq-minus-1} simplifies this expression to show that
\begin{equation}
\begin{split}
\alpha_{b1}(\nu)
&= \frac{2(n+2)\mu \cdel}{\cdelb}
\left(\frac{2(n-\beta)}{(n+2)(n+2-\beta)}\right)
\left(-\frac{1}{4}\|\nu\|^2\delta^2\right)\\
&\qquad\times
{}_2F_3\left(
1,\frac{n+2-\beta}{2};2,\frac{n+4}{2},\frac{n+4-\beta}{2};-\frac{1}{4}\|\nu\|^2\delta^2
\right)\\
&= -\mu\|\nu\|^2{}_2F_3\left(
1,\frac{n+2-\beta}{2};2,\frac{n+4}{2},\frac{n+4-\beta}{2};-\frac{1}{4}\|\nu\|^2\delta^2
\right).
\end{split}\label{eq:alpha_b1-2}
\end{equation}

The other coefficient can be computed as
\begin{equation}
\begin{split}
\alpha_{b2}(\nu) &= -\frac{(n+2)\mu \cdel}{\cdelb}
\delta^2 f''\left(-\frac{1}{4}\|\nu\|^2\delta^2\right)\\
&= -\frac{(n+2)\mu \cdel}{\cdelb}
\left(\frac{2(n-\beta)}{(n+2)(n+2-\beta)}\right)\delta^2\\
&\qquad\times{}_3F_4\left(
2,3,\frac{n+2-\beta}{2};2,3,\frac{n+4}{2},\frac{n+4-\beta}{2};-\frac{1}{4}\|\nu\|^2\delta^2
\right)\\
&= -2\mu\;{}_1F_2\left(
\frac{n+2-\beta}{2};\frac{n+4}{2},\frac{n+4-\beta}{2};-\frac{1}{4}\|\nu\|^2\delta^2
\right).
\end{split}\label{eq:alpha_b2}
\end{equation}

It is interesting to see how these formulas combine to provide a formula for the trace of the tensor $M_b$.  Since we know all eigenvalues of $M_b$, using \eqref{eq:M_s2}--\eqref{eq:alpha_b2}, we can compute the trace as 

\begin{equation*}
\begin{split}
\trace M_b(\nu) &= n\alpha_{b1}+\|\nu\|^2\alpha_{b2}
\\ &= -n\mu\|\nu\|^2\;{}_2F_3\left(
1,\frac{n+2-\beta}{2};2,\frac{n+4}{2},\frac{n+4-\beta}{2};-\frac{1}{4}\|\nu\|^2\delta^2
\right)\\
&\qquad
-2\mu\|\nu\|^2\;{}_1F_2\left(
\frac{n+2-\beta}{2};\frac{n+4}{2},\frac{n+4-\beta}{2};-\frac{1}{4}\|\nu\|^2\delta^2
\right).
\end{split}
\end{equation*}

Applying Lemma~\ref{lem:pFq-linear-combination}, cancelling the repeated $(n+4)/2$ term from the hypergeometric series coefficients, and then applying~\eqref{eq:multiplier-general} shows that
\begin{equation*}
\begin{split}
\trace M_b(\nu) &= -(n+2)\mu\|\nu\|^2\;
{}_3F_4\left(
1,\frac{n+4}{2},\frac{n+2-\beta}{2};2,\frac{n+2}{2},\frac{n+4}{2},\frac{n+4-\beta}{2};-\frac{1}{4}\|\nu\|^2\delta^2
\right)\\
&= -(n+2)\mu\|\nu\|^2\;
{}_2F_3\left(
1,\frac{n+2-\beta}{2};2,\frac{n+2}{2},\frac{n+4-\beta}{2};-\frac{1}{4}\|\nu\|^2\delta^2
\right)\\
&= (n+2)\mu \mdel(\nu).
\end{split}
\end{equation*}
This same formula can also be derived directly from~\eqref{eq:Mb}, since
\begin{equation*}
\frac{\trace\left(w\otimes w\right)}{\|w\|^{\beta+2}}
\left(\cos(\nu\cdot w)-1\right) =
\frac{\cos(\nu\cdot w)-1}{\|w\|^\beta},
\end{equation*}
yielding the integrand in~\eqref{eq:multiplier-general}.

The main results of this subsection are summarized as follows.
\begin{prop} 
\label{prop:Mb-Ms-hyperg}
The tensor multipliers $M_b$ and $M_s$ have the following hypergeometric  representations
\begin{align*}
M_b(\nu) &= -\mu\|\nu\|^2\;{}_2F_3\left(
1,\frac{n+2-\beta}{2};2,\frac{n+4}{2},\frac{n+4-\beta}{2};-\frac{1}{4}\|\nu\|^2\delta^2
\right)\;I\\
&\quad\;\;-2\mu\;{}_1F_2\left(
\frac{n+2-\beta}{2};\frac{n+4}{2},\frac{n+4-\beta}{2};-\frac{1}{4}\|\nu\|^2\delta^2
\right)\nu\otimes\nu,\\
\intertext{and,}
M_s(\nu) &= -(\lam-\mu)\,\,
{}_1F_2\left(
\frac{n+2-\beta}{2};\frac{n+2}{2},\frac{n+4-\beta}{2};-\frac{1}{4}\|\nu\|^2\delta^2
\right)^2\;\;\nu\otimes\nu.
\end{align*}
\end{prop}
An immediate consequence of this result is the convergence of the multipliers of $\LLdel$ to the multipliers of $\Nav$ in the limits as $\delta\rightarrow 0$ or as $\beta\rightarrow n+2$.
\begin{prop}
Let $\beta\le n+2$. Then
\begin{eqnarray*}
\lim_{\delta\rightarrow 0^+} M_b(\nu)&=&-\mu \|\nu\|^2\;I-2\mu\; \nu\otimes \nu,\\
\lim_{\delta\rightarrow 0^+} M_s(\nu)&=&-(\lam-\mu)\; \nu\otimes \nu,\\
\lim_{\delta\rightarrow 0^+} \Mdel(\nu)&=&-(\lam+\mu) \nu\otimes\nu -\mu \|\nu\|^2 \; I=M^{\Nav}(\nu).
\end{eqnarray*}
Moreover, let $\delta>0$. Then
\begin{eqnarray*}
\lim_{\beta\rightarrow n+2^-} M_b(\nu)&=&-\mu \|\nu\|^2\;I-2\mu\; \nu\otimes \nu,\\
\lim_{\beta\rightarrow n+2^-} M_s(\nu)&=&-(\lam-\mu)\; \nu\otimes \nu,\\
\lim_{\beta\rightarrow n+2^-} \Mdel(\nu)&=&-(\lam+\mu) \nu\otimes\nu -\mu \|\nu\|^2 \; I=M^{\Nav}(\nu).
\end{eqnarray*}
\end{prop}
\begin{proof}
This result follows from  the fact the hypergeometric functions in Proposition \ref{prop:Mb-Ms-hyperg}  are equal to $1$ under the considered limits.  
\begin{remark}
The same results hold true for the limit from above $\beta\rightarrow n+2^+$.
\end{remark}
\end{proof}
\subsubsection{Eigenvalues of the tensor multipliers}\label{sec:eigenvectors}
The form of the  multiplier $\Mdel(\nu)$ is found through  \eqref{eq:M_s2} and \eqref{eq:M_b2}, 
\begin{equation}
\Mdel(\nu) = \alpha_{b1}(\nu)I + (\alpha_{b2}(\nu)+\alpha_s(\nu))\nu\otimes\nu,\label{eq:M2}
\end{equation}
where $\alpha_{b1}$, $\alpha_{b2}$ and $\alpha_s$ are given by \eqref{eq:alpha_b1-2}, \eqref{eq:alpha_b2}, and \eqref{eq:alpha_s}, repectively.
This implies that $\Mdel$ is a real symmetric matrix. Moreover, $\nu$ is an eigenvector of $\Mdel(\nu)$,  
\begin{equation}
\Mdel(\nu)\nu=\lambda_1(\nu)\nu,\label{eq:lam_1}
\end{equation}
where 
\[
\lambda_1(\nu)=\alpha_{b1}(\nu)+(\alpha_{b2}(\nu)+\alpha_s(\nu))\|\nu\|^2.
\]
Using \eqref{eq:alpha_s}, \eqref{eq:alpha_b1-2}, and \eqref{eq:alpha_b2}, this eigenvalue, which is associated with the direction of $\nu$, has the following hypergeometric representation
\begin{equation}
\label{eq:lambda1}
\begin{split}
\lambda_1(\nu)=-\|\nu\|^2\Bigg(&
\mu\;{}_2F_3\left(
1,\frac{n+2-\beta}{2};2,\frac{n+4}{2},\frac{n+4-\beta}{2};-\frac{1}{4}\|\nu\|^2\delta^2
\right) \\
&+
2\mu\;{}_1F_2\left(
\frac{n+2-\beta}{2};\frac{n+4}{2},\frac{n+4-\beta}{2};-\frac{1}{4}\|\nu\|^2\delta^2
\right)\\
&+
(\lam-\mu)\;
{}_1F_2\left(
\frac{n+2-\beta}{2};\frac{n+2}{2},\frac{n+4-\beta}{2};-\frac{1}{4}\|\nu\|^2\delta^2
\right)^2
\Bigg).
\end{split}
\end{equation}
An alternative expression for this eigenvalue can be obtained by using Lemma~\ref{lem:pFq-linear-combination} to combine the first two hypergeometric functions, yielding

\begin{equation}
\label{eq:lambda1-alt}
\begin{split}
\lambda_1(\nu)=-\|\nu\|^2\Bigg(&
3\mu\;{}_3F_4\left(
1,\frac{5}{2},\frac{n+2-\beta}{2};2,\frac{3}{2},\frac{n+4}{2},\frac{n+4-\beta}{2};-\frac{1}{4}\|\nu\|^2\delta^2
\right) \\
&+
(\lam-\mu)\;
{}_1F_2\left(
\frac{n+2-\beta}{2};\frac{n+2}{2},\frac{n+4-\beta}{2};-\frac{1}{4}\|\nu\|^2\delta^2
\right)^2
\Bigg).
\end{split}
\end{equation}

The other $n-1$ eigenvectors are orthogonal to $\nu$. Denote by $\nu^\perp$ a vector in $\mathbb{R}^n$ orthogonal to $\nu$. Then
\begin{equation}
\Mdel(\nu)\nu^\perp=\lambda_2(\nu)\nu^\perp,\label{eq:lam_2}
\end{equation}
where 
\begin{equation}
\lambda_2(\nu)=\alpha_{b1}(\nu).\label{eq:lam2-alpha_b1}
\end{equation}
Using  \eqref{eq:alpha_b1-2}, this eigenvalue, which is associated with orthogonal directions to $\nu$, has the following hypergeometric representation
\begin{equation}
\label{eq:lambda2}
\lambda_2(\nu)=-\mu\|\nu\|^2{}_2F_3\left(
1,\frac{n+2-\beta}{2};2,\frac{n+4}{2},\frac{n+4-\beta}{2};-\frac{1}{4}\|\nu\|^2\delta^2
\right).
\end{equation}
The results in this subsection are summarized in the following.    
\begin{theorem}
\label{thm-eigenvalues-hypergeometric}
For $\nu\in \Rn$, the eigenvalue $\lambda_1(\nu)$ of $\Mdel(\nu)$, associated with the direction of $\nu$, is given by the hypergeometric representation \eqref{eq:lambda1-alt} and 
the eigenvalue $\lambda_2(\nu)$ of $\Mdel(\nu)$, associated with orthogonal directions to $\nu$, is given by
the hypergeometric representation \eqref{eq:lambda2}.
\end{theorem}
\begin{cor}
\label{cor:eigenvalues-converge}
Let $\nu\in\Rn$. Then, the tensor multipliers $\Mdel(\nu)$ and $M^{\Nav}(\nu)$ have the same set of eigenvectors: $\nu$ and $n-1$ eigenvectors orthogonal to $\nu$. Moreover, the eigenvalues of $\Mdel(\nu)$ converge to the eigenvalues of $M^{\Nav}(\nu)$ in the local limits as follows: for $\beta\le n+2$,
\begin{align*}
\lim_{\delta\rightarrow 0^+} \lambda_1(\nu)&=-(\lam+2\mu)\|\nu\|^2,\\
\lim_{\delta\rightarrow 0^+} \lambda_2(\nu)&=-\mu \|\nu\|^2,\\
\intertext{ and for $\delta>0,$}
\lim_{\beta\rightarrow n+2^-} \lambda_1(\nu)&=-(\lam+2\mu)\|\nu\|^2,\\
\lim_{\beta\rightarrow n+2^-} \lambda_2(\nu)&=-\mu\|\nu\|^2.
\end{align*}
\end{cor}
\subsubsection{Integral representations for the eigenvalues of the peridynamic multipliers}
\label{sec:eigenvalues-integral}
In this section, we provide  integral representations for the eigenvalues $\lambda_1$ and $\lambda_2$, given by \eqref{eq:lambda1} and \eqref{eq:lambda2}, respectively.
\begin{theorem}
\label{thm-eigenvalues-integral}
The eigenvalue of $\Mdel(\nu)$ associated with the direction of $\nu$ is given by
\begin{eqnarray}
\nonumber\lambda_1(\nu) &=& (n+2)\mu \cdel
\int_{B_\delta(0)}\frac{(\nu\cdot w)^2}{\|\nu\|^2 \|w\|^{\beta+2} }(\cos(\nu\cdot w)-1)\;dw \\
&&\hspace{1.5in}-(\lam-\mu)\left(\frac{\cdel}{2}
\int_{B_\delta(0)}\frac{\nu\cdot w}{\|\nu\|\|w\|^\beta }\sin(\nu\cdot w)\;dw\right)^2. \label{eq:lam1}
\end{eqnarray}
The eigenvalue of $\Mdel(\nu)$ associated with orthogonal directions to $\nu$ is given by
\begin{eqnarray}
\lambda_2(\nu) &=& (n+2)\mu \cdel
\int_{B_\delta(0)}\frac{\nu\cdot w}{\|\nu\|^2\|w\|^{\beta+2} }(\sin(\nu\cdot w)-\nu\cdot w)\;dw.\label{eq:lam2}
\end{eqnarray}
\end{theorem}
\begin{proof}
Solving  \eqref{eq:lam_1} for  $\lambda_1(\nu)$ we obtain
\begin{equation}
\lambda_1(\nu)= M(\nu)\nu\cdot\frac{\nu}{\|\nu\|^2}.\label{eq:lam1def}
\end{equation}
The result  \eqref{eq:lam1} follows from \eqref{eq:lam1def} combined with the fact that $M=M_b+M_s$ and the integral representations of the multipliers given in \eqref{eq:Mb} and \eqref{eq:Ms}.

To derive \eqref{eq:lam2}, we use \eqref{eq:lam2-alpha_b1} and \eqref{eq:alpha_b1-1} to write
\begin{equation}
\lambda_2(\nu)=\frac{(n+2)\mu \cdel}{\cdelb}\left(
2\;{}_1F_2\left(
\frac{n-\beta}{2};\frac{n+2}{2},\frac{n+2-\beta}{2};-\frac{1}{4}\|\nu\|^2\delta^2
\right)
-2\right).\label{eq:lam2-1F2}
\end{equation}
From \eqref{eq:dm-dnu}, \eqref{eq:pd-m-f}, and \eqref{eq:alpha_s}, we obtain
\[
\cdel\int_{B_\delta(0)}\frac{w}{\|w\|^\beta}\sin(\nu\cdot w)\;dw = 2\;{}_1F_2\left(
\frac{n+2-\beta}{2};\frac{n+2}{2},\frac{n+4-\beta}{2};-\frac{1}{4}\|\nu\|^2\delta^2
\right)\nu.
\]
Thus, by replacing $\beta$ by $\beta+2$, then multiplying both sides of the above equation by $\cdot\frac{\nu}{\|\nu\|^2}$, we find
\begin{equation}
\cdelb\int_{B_\delta(0)}\frac{\nu\cdot w}{\|w\|^{\beta+2}\|\nu\|^2}\sin(\nu\cdot w)\;dw = 2\;{}_1F_2\left(
\frac{n-\beta}{2};\frac{n+2}{2},\frac{n+2-\beta}{2};-\frac{1}{4}\|\nu\|^2\delta^2
\right).\label{eq:2-1F2}
\end{equation}
Using \eqref{eq:identity}, we have the following identities
\begin{eqnarray}
\nonumber
2&=& (2 I) \nu \cdot\frac{\nu}{\|\nu\|^2}\\
\nonumber
&=& \left(\cdelb
\int_{B_\delta(0)}\frac{w\otimes w}{\|w\|^{\beta+2}}\;dw\right) \nu \cdot\frac{\nu}{\|\nu\|^2}\\
&=& \cdelb
\int_{B_\delta(0)}\frac{(\nu\cdot w)^2}{\|w\|^{\beta+2}\|\nu\|^2}\;dw.\label{eq:2}
\end{eqnarray}
Using \eqref{eq:lam2-1F2}, \eqref{eq:2-1F2} and \eqref{eq:2}, we obtain
\[
\lambda_2(\nu)=\frac{(n+2)\mu \cdel}{\cdelb}\left(\cdelb\int_{B_\delta(0)}\frac{\nu\cdot w}{\|w\|^{\beta+2}\|\nu\|^2}(\sin(\nu\cdot w)-\nu\cdot w)\;dw\right),
\]
from which the result follows.
\end{proof}

\section{Eigenvalues of the linear peridynamic operator}\label{sec:periodic}
Consider $\LLdel$ as an operator on the periodic torus
\begin{equation*}
\T^n = \prod_{i=1}^n[0,\ell_i],\qquad\text{with }\ell_i>0,\quad i=1,2,\ldots,n.
\end{equation*}
In  this section, we use the multiplier approach developed in the previous sections to identify the eigenvalues and the eigenvector fields of the operator $\LLdel$,
\[
\LLdel \boldsymbol \phi = \lambda \boldsymbol \phi.
\] 

Let $\gamma$ be a fixed vector in $\Rn$. For any $k \in\mathbb{Z}^n$, define
\begin{eqnarray}\label{eq:nu}
\nonumber
\nu_k&=&(2\pi k_1/ \ell_1,2\pi k_2/ \ell_2,\ldots,2\pi k_n/ \ell_n)^T,\\
\nonumber
\boldsymbol\psi_k(x) &=& e^{i\nu_k\cdot x} \gamma.
\end{eqnarray}
Then, by applying $\LLb$ in \eqref{eq:Lb}, and using \eqref{eq:Mb}, we obtain
\begin{eqnarray}\label{eq:Lb-psi}
\nonumber
\LLb\boldsymbol\psi_k(x) &=& (n+2)\,\mu\,\cdel\int_{B_\delta(0)}
   \frac{w\otimes w}{\|w\|^{\beta+2}}\big( e^{i\nu_k\cdot(x+w)}\,\gamma-e^{i\nu_k\cdot x}\,\gamma\big)\,dw,\\
   \nonumber
   &=& \left((n+2)\,\mu\,\cdel\int_{B_\delta(0)}
   \frac{w\otimes w}{\|w\|^{\beta+2}}\big( e^{i\nu_k\cdot w}-1\big)\,dw\right) \,\,e^{i\nu_k\cdot x}\,\gamma,\\
   &=& M_b(\nu_k)\boldsymbol\psi_k(x).
\end{eqnarray}
Similarily, by applying $\LLs$ in \eqref{eq:Ls}, and using \eqref{eq:Ms}, we obtain
\begin{eqnarray}\label{eq:Ls-psi}
\nonumber
\LLs\boldsymbol\psi_k(x) &=& (\lam-\mu)\,\frac{(\cdel)^2}{4}\int_{B_\delta(0)}\int_{B_\delta(0)}
 \frac{w}{\|w\|^\beta}
\otimes\frac{q}{\|q\|^\beta}\,e^{i\nu_k\cdot(x+q+w)}\,\gamma\,dq dw,\\
   \nonumber
   &=& (\lam-\mu)\,\frac{(\cdel)^2}{4}\left(\int_{B_\delta(0)}
 \frac{w}{\|w\|^\beta} e^{i \nu_k\cdot w}\,dw\right)
\otimes\left(\int_{B_\delta(0)}\frac{q}{\|q\|^\beta} e^{i \nu_k\cdot q}\,dq\right)\,\,e^{i\nu_k\cdot x}\,\gamma,\\
   &=& M_s(\nu_k)\boldsymbol\psi_k(x).
\end{eqnarray}
Equations \eqref{eq:Lb-psi}, \eqref{eq:Ls-psi}, and \eqref{eq:M} yield
\begin{eqnarray}\label{eq:LLdel-psi}
\nonumber
\LLdel \boldsymbol\psi_k&=&\LLb\boldsymbol\psi_k+\LLs\boldsymbol\psi_k,\\
\nonumber
&=&M_b(\nu_k)\boldsymbol\psi_k+M_s(\nu_k)\boldsymbol\psi_k,\\
&=&\Mdel(\nu_k)\boldsymbol\psi_k.
\end{eqnarray}

Denote by $ M_k=\Mdel(\nu_k)$, the tensor multiplier evaluated at $\nu_k$. From \eqref{eq:M2}, \eqref{eq:lam_1}, and \eqref{eq:lam_1}, the matrix $M_k$ is real symmetric and has $n$  orthogonal eigenvectors: one in the direction of $\nu_k$, and $n-1$ orthogonal to $\nu_k$, denoted by  
$\zeta_{k}^2, \ldots, \zeta_{k}^n$. The associated eigenvalues are denoted by $\lambda_{1}(k):=\lambda_1(\nu_k)$ and $\lambda_{2}(k):=\lambda_2(\nu_k)$, respectively. Explicitly,
\begin{eqnarray}\label{eq:Mk-eigenvalues}
\nonumber
M_k \nu_{k}&=&\lambda_{1}(k) \nu_{k},\\
\nonumber
M_k \zeta_{k}^j&=&\lambda_{2}(k) \zeta_{k}^j,\quad j=2,\dots,n.
\end{eqnarray}
Define 
\begin{eqnarray}
\label{eq:eigenvector-fields-1}
\boldsymbol\phi_{k}^1(x)&=&e^{i\nu_k\cdot x} \nu_k,\\
\label{eq:eigenvector-fields-2}
\boldsymbol\phi_{k}^j(x)&=&e^{i\nu_k\cdot x} \zeta_{k}^j,\quad j=2,\dots,n.
\end{eqnarray}
Then,
\begin{eqnarray}\label{eq:LLdel-eigenvalues-1}
\nonumber
\LLdel \boldsymbol\phi_{k}^1&=&\Mdel(\nu_k)\boldsymbol\phi_{k}^1,\\
\nonumber
&=& e^{i\nu_k\cdot x} M_k \nu_{k},\\
\nonumber
&=& e^{i\nu_k\cdot x} \lambda_{1}(k) \nu_{k},\\
&=& \lambda_{1}(k) \boldsymbol\phi_{k}^1,
\end{eqnarray}
and, for $j=2,\ldots,n$,
\begin{eqnarray}\label{eq:LLdel-eigenvalues-j}
\nonumber
\LLdel \boldsymbol\phi_{k}^j&=&\Mdel(\nu_k)\boldsymbol\phi_{k}^j,\\
\nonumber
&=& e^{i\nu_k\cdot x} M_k \zeta_{k}^j,\\
\nonumber
&=& e^{i\nu_k\cdot x} \lambda_{2}(k) \zeta_{k}^j,\\
&=& \lambda_{2}(k) \boldsymbol\phi_{k}^j.
\end{eqnarray}

This shows that the eigenvalues $\lambda_1(k)$ and $\lambda_2(k)$ of the peridynamic operator $\LLdel$, defined on the periodic torus, are the eigenvalues of the tensor multipliers $M_k$.
The summary is given in the following result.

\begin{theorem}
\label{thm-eigenvalues-L}
Let $k\in\mathbb{Z}^n$. 
The eigenvalues of the linear peridynamic operator $\LLdel$ are $\lambda_1(k):=\lambda_1(\nu_k)$ with associated eigenvector fields $\boldsymbol\phi_{k}^1$ and $\lambda_2(k):=\lambda_2(\nu_k)$ with associated eigenvector fields $\boldsymbol\phi_{k}^j$, for $j=2,\ldots,n$, where $\lambda_1$ and $\lambda_2$ are given in Theorem \ref{thm-eigenvalues-hypergeometric}, or equivalently, in Theorem \ref{thm-eigenvalues-integral}. The eigenvector fields   $\{\boldsymbol\phi_{k}^j\}_{k\in\mathbb{Z}^n}$, with $j=1,\ldots,n$, are defined in \eqref{eq:eigenvector-fields-1} and \eqref{eq:eigenvector-fields-2}.
\end{theorem}
Convergence of the eigenvalues of the peridynamic operator to the eigenvalues of the Navier operator follows immediately from  Corollary \ref{cor:eigenvalues-converge} and Theorem \ref{thm-eigenvalues-L}. The eigenvalues of the Navier operator in \eqref{Nhomo} are given  by
\begin{eqnarray}
\label{N-eigenvalues}
    \lambda^{\Nav}_1(k)&=& -(\lam+2\mu)\|k\|^2,\quad \mbox{ and }\quad
    \lambda^{\Nav}_2(k)= -\mu\;\|k\|^2. 
\end{eqnarray}
\begin{figure}
\centering
\includegraphics[width=\textwidth]{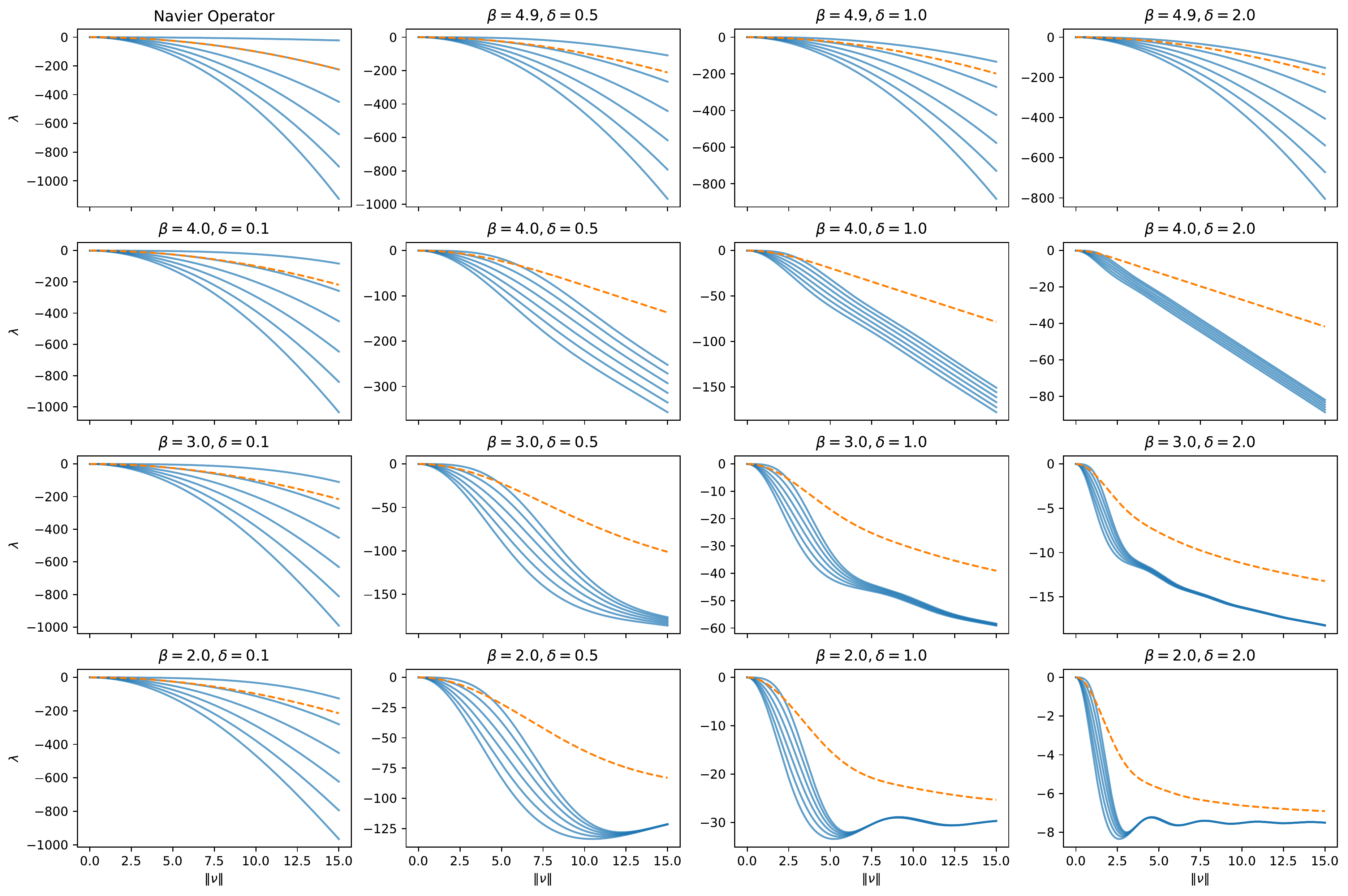}
\caption{The eigenvalues (vertical axis) $\lambda_1(\nu)$ and $\lambda_2(\nu)$, as given by \eqref{eq:lam1} and \eqref{eq:lam2}, for the 3D case. Here $\|\nu\|$ (horizontal axis) is sampled at $1000$ equispaced points in the interval $[0,15]$ and $\delta$ and $\beta$ are as given in the titles.  The shear modulus and the second Lam\'{e} parameter are given by $\mu=1$ and $\lam=-1.9, -1, 0, 1, 2$. For each plot, the dashed line shows $\lambda_2(\nu)$ and the solid lines show $\lambda_1(\nu)$ corresponding to the different values of $\lam$ in a decreasing order, i.e., the top solid curve corresponds to $\lam=-1.9$, the second corresponds to $\lam=-1$, etc.}
\label{fig:vector-multipliers}
\end{figure}

\section{Discussion}
The hypergeometric representations of the eigenvalues, given in \eqref{eq:lambda1} and \eqref{eq:lambda2}, are utilized to compute the eigenvalues $\lambda_1$ and $\lambda_2$ as shown in Figure \ref{fig:vector-multipliers}. 
It easily follows from \eqref{N-eigenvalues} that for the Navier operator,   
the eigenvalues are non-positive and  $\lambda^{\Nav}_1(\nu)$ is decreasing in $\lam$ for a fixed value of $\mu$, which additionally can  be seen from the eigenvalues' curves for the Navier operator (top-left) in Figure  \ref{fig:vector-multipliers}. The non-positivity of the eigenvalues and the  monotonicity of $\lambda_1(\nu)$ as a function of $\lam$  hold true as well for the peridynamic operator. These observations follow from \eqref{eq:lam1} and \eqref{eq:lam2} for any $\delta>0$ and $\beta<n+2$ and can also be observed   in Figure  \ref{fig:vector-multipliers}. In addition, in this figure, we note that in the first row (which corresponds to $\delta$ being close to $0$) and the first column (which corresponds to $\beta$ being close to $n+2$) the eigenvalues satisfy $\lambda_1(\nu)\approx \lambda_1^\Nav(\nu)$ and $\lambda_2(\nu)\approx \lambda_2^\Nav(\nu)$, which is consistent with Corollary \ref{cor:eigenvalues-converge} and the fact that the hypergeometric functions in \eqref{eq:lambda1-alt} and \eqref{eq:lambda2} are continuous.  
Moreover, in the  second row of Figure  \ref{fig:vector-multipliers}, which corresponds to $\beta=n+1$, we observe that the curves of the eigenvalues $\lambda_1(\nu)$ and $\lambda_2(\nu)$ are linear, of order $~\|\nu\|$,  for large values of $\|\nu\|$.  The asymptotic behavior of  the eigenvalues in the third row of this figure, which corresponds to $\beta=n$,  can be shown to be logarithmic in $\|\nu\|$.
Furthermore, for integrable kernels (when $\beta<n$) it can be seen from the fourth row of  Figure  \ref{fig:vector-multipliers} that the eigenvalues are bounded.
These observations can be rigorously proved, similar to the approach followed in \cite{alali2019fourier}, using the hypergeometric representations \eqref{eq:lambda1-alt} and \eqref{eq:lambda2}, and can be used to derive regularity results for peridynamic equations. Lastly, we notice in the figure that the curves of $\lambda_1(\nu)$ for different values of $\lam$ converge to a single curve for large values of $\|\nu\|$, in the case that $\beta<n$.
 This can be shown using the integral representation  \eqref{eq:lam1} of $\lambda_1$ and the Riemann–Lebesgue Lemma, which implies that $\sin(\nu\cdot w)$ weakly converges to $0$ in the limit as $\|\nu\|\rightarrow\infty$.

\bibliographystyle{acm}
\bibliography{refs}
\end{document}